\documentclass[12pt]{amsart}
\usepackage{a4wide}
\usepackage[T1]{fontenc}
\usepackage{amssymb,amsmath,amsthm,latexsym}
\usepackage{mathrsfs}
\usepackage[usenames,dvipsnames]{color}
\usepackage{euscript}
\usepackage{graphicx}
\usepackage{mdwlist}
\usepackage{enumerate}
\usepackage{mathtools,dsfont,wasysym}
\usepackage{stmaryrd}
\usepackage{tikz-cd}
\usepackage{hyperref}
\hypersetup{colorlinks=true,  linkcolor=blue, citecolor=red, urlcolor=cyan}

\newtheorem{theorem}{Theorem}[section]
\newtheorem{lemma}[theorem]{Lemma}
\newtheorem{corollary}[theorem]{Corollary}
\newtheorem{proposition}[theorem]{Proposition}

\newcounter{maintheorem}

\theoremstyle{remark}
\newtheorem{remark}[theorem]{Remark}

\theoremstyle{definition}
\newtheorem{definition}[theorem]{Definition}

\numberwithin{equation}{section}
\makeatother

\newcommand{\R}{\mathbb{R}}
\renewcommand{\H}{\mathbb{H}}
\newcommand{\N}{\mathbb{N}}
\newcommand{\e}{\varepsilon}

\newcommand{\nn}[1]{{\left\vert\kern-0.25ex\left\vert\kern-0.25ex\left\vert #1 
		\right\vert\kern-0.25ex\right\vert\kern-0.25ex\right\vert}}

\newcommand{\F}{\mathcal{F}}

\newcommand{\Lip}{\mathrm{Lip}}
\renewcommand{\leq}{\leqslant}
\renewcommand{\geq}{\geqslant}

\renewcommand{\d}{\mathrm{d}}
\renewcommand{\div}{\mathrm{div}}
\newcommand{\essup}{\mathop{\mathrm{essup}}}
\let\phi\varphi
\renewcommand{\S}{\mathbb{S}}

\newcounter{smallromans}

	{\end{list}}

\def\XXint#1#2#3{{\setbox0=\hbox{$#1{#2#3}{\int}$}
		\vcenter{\hbox{$#2#3$}}\kern-.5\wd0}}

\begin{document}
	\title[Lipschitz-Free Spaces over Riemannian Manifolds]{Isometric Representation of Lipschitz-Free Spaces\\over Connected Orientable Riemannian Manifolds}
	\author{Franz Luggin}
	\address{Universit\"{a}t Innsbruck, Department of Mathematics, Technikerstra\ss e 13, 6020 Innsbruck, Austria}
	\email{Franz.Luggin@student.uibk.ac.at}	
	\thanks{}
	
	\date{}

	\begin{abstract}
		We show that the Lipschitz-Free Space over a connected orientable $n$-dimensional Riemannian manifold $M$ is isometrically isomorphic to a quotient of $L^1(M,TM)$, the integrable sections of the tangent bundle $TM$, if $M$ is either complete or lies isometrically inside a complete manifold $N$. Two functions are deemed equivalent in this quotient space if their difference has distributional divergence zero.
		
		This quotient is the pre-annihilator of the exact essentially bounded currents, and if $M$ is simply connected, one may replace ``exact'' with ``closed'' currents.
	\end{abstract}

	\maketitle
	
	\section{Introduction}
	\subsection{Lipschitz Spaces}
	Given a metric space with a distinguished point $0_M\in M$, let $\Lip_0(M)$ be the space of all functions $F\colon M\to \R$ which are Lipschitz on $M$ and satisfy $F(0_M)=0$. When equipped with the usual addition and scalar multiplication of functions, this space is a vector space on which the Lipschitz constant $\Lip(F):=\sup_{x\neq y}\frac{|F(x)-F(y)|}{d(x,y)}$ is a norm, $\|F\|_\Lip:=\Lip(F)$.
	We can now define the Lipschitz-Free space $\F(M)$ as a predual of $\Lip_0(M)$ (under certain conditions the (strongly) unique one, see e.g.\@ \cite[Thm. 3.26-28]{Weaver}). To do this, define $\F(M):=\overline{\mathrm{span}\{\delta(x)\colon x\in M\}}$, where $\delta(x)$ is an evaluation functional such that $\langle F,\delta(x)\rangle=F(x)$ for all $F\in\Lip_0(M)$, and for all $x\in M$.
	
	These spaces have been a topic of intense research in recent years, starting with \cite{GK2003} and the first edition of \cite{Weaver}, and \cite{Godefroy-survey} has sparked renewed interest in it. A non-exhaustive selection of relevant papers from the last 10 years include \cite{CCGMR2019, CJ2017, D2015, FG2022, GL2018}. Strong results are available in special cases, as e.g.\@ when $M$ is a Banach space, see Kaufmann \cite{K2015}. Moreover, Ostrovska and Ostrovskii studied the isometric structure of Lipschitz-Free spaces over finite metric spaces in \cite{OO2022}, and investigated which Lipschitz-Free spaces contain isometric copies of $\ell^1$ in \cite{OO2019, OO2020, OO2024}.
	
	In general, isomorphisms or isometries of $\F(M)$ to well-known spaces allow for the transfer of isomorphic and isometric properties, respectively, like approximation properties, existence of Schauder bases, etc. For example, recent results include Gartland in \cite{Gartland} showing that $\F(\H^d)\cong\F(\R^d)$ isomorphically, and Albiac, Ansorena, C\'uth, Doucha in \cite{AACD} showing $\F(\mathbb S^d)\cong\F(\R^d)$, again, isomorphically.
	
	The main result of this paper is that the Lipschitz space $\Lip_0(M)$ of a connected Riemannian manifold $M$ is isometrically isomorphic to the exact $L^\infty$-currents on $M$
	\begin{align*}
		\Lip_0(M)\equiv \{f \in L^\infty(M,T^*M)\colon \exists g\colon M\to\R\colon f=\d g\}
	\end{align*}
	via an isometry that is weak$^*$-to-weak$^*$ continuous. In the particular case when $M$ is simply connected, by Corollary~\ref{cor:closed-currents}, this is equivalent to the space of closed $L^\infty$ currents (all $f\in L^\infty(M,T^*M)$ such that $\d f=0$). This is useful since, in general, closedness of a form is much easier to verify than exactness.
	
	As corollaries of that, if $M$ is complete, the Lipschitz-Free space $\F(M)$ is isometric to $L^1(M,TM)/\{g\in L^1(M,TM)\colon \div(g)=0\}$, the integrable sections modulo the equivalence relation $f\sim g\iff \div(f-g)=0$ in the sense of distributions on $M$ by Theorem~\ref{thm:isometry-of-lip-free}; whereas if $M$ lies isometrically inside a larger complete manifold $N$, then the same result holds with a slightly modified quotient consisting of all integrable sections whose extension to $N$ by zero has distributional divergence zero: $\{g\in L^1(M,TM)\colon\exists \hat g\in L^1(N,TN)\colon \hat g|_M=g,\hat g|_{N\setminus M} = 0, \div(\hat g)=0\}$ (by Corollary~\ref{cor:domain-inside-complete}).
	
	Complete connected Riemannian manifolds include the Euclidean spaces $\R^n$, the hyperbolic spaces $\H^n$ and the spheres $\S^n$, so this result covers them and all open connected and geodesically convex subsets thereof.
	
	This is a generalization of the result by C\'uth, Kalenda and Kaplick\'y in \cite{CKK2016}, where the same is shown for nonempty convex open subsets of $\R^n$ with the metric induced by any norm on $\R^n$, and follows roughly the same structure, whereas Ostrovska and Ostrovskii discuss the isometric structure of Lipschitz-Free spaces over finite metric spaces in \cite{OO2022}.
	
	A different generalization by Flores was recently published as a preprint in \cite{F2024}, where $M$ is instead a domain in a finite-dimensional normed space $E$ equipped with the intrinsic metric. If $M$ lies in a Riemannian manifold, this result follows from the one present in this paper (see Remark \ref{rem:comparison-domain-results}), since in Flores' theorem, $M$ as a whole forms some domain within a normed space $E$, with the metric on $M$ coinciding with the induced intrinsic metric from $E$, and of course, since $E$ is a finite-dimensional normed space, it is complete.
	
	\subsection{Theory of Distributions on Smooth Manifolds}
	This paper assumes familiarity with the basic theory of Riemannian manifolds, working with smooth charts and smooth $n$-forms, the tangent and cotangent spaces, as well as the Lebesgue measure on Riemannian manifolds, and how it can be used to define integrals and $L^1(M)$. For more information on these topics, see e.g.\@ \cite{AE} and \cite{Dieu}. In particular, the existence of a volume form $\d V$ on $M$ that corresponds locally to a $\mathcal C^\infty$-multiple of the Lebesgue measure on $\R^n$ is proven in Theorem 16.22.2 in \cite[p.~163]{Dieu}.
	
	For anything going beyond this `classical' theory, including practically any non-smooth object on the manifold, we follow the book \cite{GKOS}, which introduces, for a smooth manifold $M$, the space of (compactly supported) smooth $k$-forms $\Omega^k(M)$ ($\Omega_c^k(M)$), and then defines the \emph{distributional $k$-forms} $\Omega^k(M)^\diamond$ as the dual of $\Omega^{n-k}_c(M)$ (where $n$ is the dimension of $M$), with regular objects embedding from $\Omega^k(M)$ into $\Omega^k(M)^\diamond$ via integration:
	\begin{align}
		\forall \omega\in\Omega^k(M)\colon \langle \omega, \tau\rangle_{(\Omega^k(M)^\diamond,\Omega^{n-k}_c(M))} := \int_M\omega\wedge\tau.\label{eq:dual-pairing}
	\end{align}
	For example, one can see that for $0$-forms, while $\Omega^0(M)=\mathcal C^\infty(M)$, $L^1_{\mathrm{loc}}(M)\subseteq\Omega^0(M)^\diamond$, since those are exactly the functions which are integrable on compact subsets of $M$.
	
	It turns out that these spaces of distributional $k$-forms are well-behaved in a couple of key ways, namely $\overline{\Omega^k(M)}^{wsc}=\Omega^k(M)^\diamond$ with respect to the dual pairing (\ref{eq:dual-pairing}), and the exterior derivative $\d$ and the Lie derivative $L_\phi$ (for a smooth vector field $\phi$) have unique continuous extensions from the spaces of smooth forms to their respective distributional forms which preserve all their most crucial properties \cite[Thm.~3.1.18, 3.1.23, 3.1.24]{GKOS}.
	
	Importantly, observe that $\Lip_0(M)\subset\Omega^0(M)^\diamond$ since Lipschitz functions are locally integrable.
	
	Lastly, the following lemma follows easily from results in the literature:
	\begin{lemma}\label{lem:distribution-is-abs-cont-function}
		If a distribution $h\in\mathcal D'(M)$ defined on an $n$-dimensional Riemannian manifold $M$ has all its partial derivatives $\partial_j h$ lie in $L^p_{loc}(M)$ for some $n<p<\infty$ and $\d h$ is (a regular distribution induced by) a $1$-form, then $h$ is an absolutely continuous function.
	\end{lemma}
	
	This follows from a combination of two theorems, firstly Theorem 4.5.12 in Hörmander \cite[p.~123]{Hoer} which states that if $X\subset\R^n$ open and $\partial_j h\in L^p_{loc}(X)$ for $p>n$ and all $1\leq j\leq n$, then $h$ is induced by a (even Hölder-continuous) function. And secondly, according to Schwartz in \cite[Thm.~XVIII]{Schwartz}, if all derivatives of rank $\leq1$ of $h$ are functions, then $h$ is an absolutely continuous function.
	
	\section{Preliminaries}
	\subsection{Derivative of Lipschitz Functions}
	For the remainder of this paper, let $M$ be a connected orientable $n$-dimensional Riemannian manifold without boundary, with volume form $\d V$. Note that due to a result in appendix A of \cite{Spivak}, this in particular implies that $M$ is second-countable (since there is only one connected component and therefore the Riemannian metric tensor induces a `genuine' metric space, with no points of infinite distance).
	
	In order to connect the space of Lipschitz functions to the space of essentially bounded functions $L^\infty$, if we are in $\R^n$, we have Rademacher's theorem, ensuring almost-everywhere differentiability. On a Riemannian manifold $M$, it is not too difficult to achieve the same result, and while said result is widely known, it is included for the sake of completeness.
	
	\begin{proposition}\label{prop:Elementary_Calculus}
		Let $F\colon M\to \R$ be an $L$-Lipschitz function. Then the following hold:
		\begin{enumerate}[i)]
			\item For almost all $x\in M$, the differential $\d F(x)\in T_x^*M$ of $F$ exists and satisfies $\|\d F(x)\|\leq L$.\label{list:Elementary_Calculus1}
			\item The mapping $\d F\colon M\to T^*M\colon x\mapsto (x,\d F(x))$ is well-defined as an $L^\infty(M,T^*M)$ section.
		\end{enumerate}
	\end{proposition}
	
	\begin{proof}
		\begin{enumerate}[i)]
			\item Let $(\phi_m,U_m)_{m\in\N}$ be a countable locally finite atlas of $M$. Let furthermore $\phi_m^{-1}$ be Lipschitz (since they are smooth, this can be achieved by simply shrinking the domains $U_m$ by an arbitrarily small amount, in such a way that they still overlap).
			
			Then, $F\circ \phi_m^{-1}$ will be a Lipschitz map from $V_m:=\phi_m(U_m)\subset\R^n$ to $\R$, and thus the classical Rademacher theorem yields that $F\circ \phi_m^{-1}$ is a.e.\@ differentiable.
			
			So, as the composition of an a.e.\@ differentiable and a smooth function, $F|_{U_m}=(F\circ\phi_m^{-1})\circ\phi_m$ will also be a.e.\@ differentiable. And clearly, the values on the different sets $U_m$ are compatible, since $F|_{U_k}(x)=F|_{U_m}(x)$ for all $x\in U_k\cap U_m$.
			
			Hence $F$ is differentiable on at least all points that do not fall into any of the countably many null sets on which the restrictions to the sets $U_m$ are non-differentiable, in other words, $F$ is a.e.\@ differentiable itself.
			
			Thus, by the explanation following Theorem 3.1.23 in \cite{GKOS}, the distributional derivative $\d F$ defined in said theorem coincides a.e.\@ with the (a.e.)\@ classical derivative of $F$ found via Rademacher.
			
			This will be used to show $\|\d F(x)\|\leq L$ in all points of differentiability $x$ of $F$ since, for any such point $x$, we now know that we can find $\|\d F(x)\|$ classically in a variety of ways, for example by looking at the set $\Gamma(x)$ of all piecewise continuously differentiable, piecewise unit-speed (thus $1$-Lipschitz) curves on $M$ starting in $x$, and calculating
			\begin{align*}
				\|\d F(x)\|=\sup_{\gamma\in \Gamma(x)}|(F\circ\gamma)'(0)|=\sup_{\gamma\in \Gamma(x)}\lim_{t\searrow0}\frac{|F(\gamma(t))-F(\gamma(0))|}{t}\leq \sup_{\gamma\in \Gamma(x)}\lim_{t\searrow0}L=L.
			\end{align*}
			
			\item The set of non-differentiability points is of measure zero, and for all $x\in M$ outside of that null set, $\d F(x)\in T_x^*M$, so $\d F$ is well-defined.\qedhere
		\end{enumerate}
	\end{proof}
	
	\subsection{Essentially Bounded Sections}
	Let $M$ be a Riemannian manifold and $E$ a vector bundle over $M$ with projection $\pi$. Then, any measurable right-inverse of $\pi$, i.e.\@ any measurable function $f\colon M\to E$ such that $\pi\circ f=\mathrm{id}_M$, is called a \emph{section} of $E$.
	
	Smooth $1$-forms are by definition smooth sections of the cotangent bundle, which we denote by $\Gamma(M,T^*M)$ (following \cite{GKOS}). On the other hand, using the Lebesgue measure on $M$, we can define the space $L^\infty(M,T^*M)$ of \emph{essentially bounded sections} of the cotangent bundle (not necessarily smooth ones) as the Banach space of all sections which have essentially bounded supremum:
	\begin{align*}
		\|f\|_{L^\infty}:=\inf_{N\subset M \text{ null set}}\sup_{x\in M\setminus N}\|f(x)_2\|_2<\infty.
	\end{align*}
	Here, $f(x)_2\in T_x^*M\cong\R^n$ denotes the second component of $f(x)=(x,f(x)_2)$.
	It follows that $L^\infty(M,T^*M)\subset\Omega^1(M)^\diamond$.

	\subsection{Integrable Sections}
	We will now define the Banach space $L^1(M,TM)$, the space of equivalence classes of sections that are integrable with respect to the canonical volume form $\d V$ of $M$, and then show that this is indeed a predual of $L^\infty(M,T^*M)$.
	
	To do that, we will however first need to define:
	\begin{definition}
		A sequence of pairs $(V_m,\phi_m)_{m\in\N}$, where for every $m\in\N$, $V_m\subseteq M$ is open and $\phi_m\colon V_m\to\R^n$ is a chart, is called a \emph{patchwork of trivializations} if for every $V_m$ there exists an injective tangent space trivialization $\Phi_m\colon TV_m\to\R^n$ and, in addition, $M=\bigcup_m\overline{V_m}$, the domains $V_m$ are all pairwise disjoint and $N:=M\setminus\bigcup_mV_m$ is a null set.
	\end{definition}
	
	\begin{lemma}
		A patchwork of trivialisations exists on any separable smooth manifold $M$.
	\end{lemma}
	
	\begin{proof}
		Start by defining a smooth atlas $(U_p,\phi_p)_{p\in M}$ where $U_p$ is an open ball around $p\in M$ such that there exists an injective tangent space trivialization $\Phi_p\colon TU_p\to\R^n$. There must exist some neighbourhood $U$ of $p$ such that $\Phi\colon TU\to\R^n$ is injective since we are in a smooth manifold.
		
		Then, by $\sigma$-compactness of $M$, there exists a countable subcover $(U_{p_m},\phi_{p_m})_{m\in\N}$.
		Then, we let:
		\begin{gather*}
			V_0:=U_{p_0}\quad\text{and}\quad V_{m+1} := U_{p_{m+1}}\setminus\bigcup_{k=0}^m\overline{V_k}.
		\end{gather*}
		We cut away the overlaps, or rather, their closures, so the resulting sets remain open.
		
		Note that $N$ is a subset of the countable union of spheres $\partial U_{p_m}$, each of which is Lebesgue-null. Thus, if we restrict $\phi_{p_m}$ to $V_m$ and call that new chart $\phi_m$, then $(V_m,\phi_m)_{m\in\N}$ is a patchwork of trivializations.
	\end{proof}
	
	Now, to get back to defining integrable sections: clearly, on each chart $\phi_m$, $f|_{V_m}\colon V_m\to TV_m$ is integrable iff $\Phi_m\circ f\circ\phi_m^{-1}\colon \phi_m(V_m)\subset\R^n\to\R^n$ is integrable with respect to a specific measure $\mu_m$ that depends on the metric tensor $g$ of $M$ and the chart $\phi_m$, and is absolutely continuous with respect to the Lebesgue measure. Similarly:
	\begin{align}
		\int_Mf\d V &= \int_{M\setminus N}f\d V=\sum_m\int_{V_m}f\d V\notag\\
		&= \sum_m\int_{\phi_m(V_m)} \Phi_m\circ f\circ \phi_m^{-1} \underset{=:\d\mu_m}{\underbrace{\sqrt{\left|\det\left[g\left(\frac{\partial\phi_m}{\partial x_i},\frac{\partial\phi_m}{\partial x_j}\right)\right]_{i,j=1}^n\right|}\d\lambda}}. \label{eq:volume-measure-Rn}
	\end{align}
	
	Now we will use that the regions $V_m$ are the domain of a single chart $\phi_m$ to define a specific orthonormal frame $(x_i)_{i=1}^n$ almost everywhere, on $M\setminus N$. We do this by first defining local sections $x_i^m\colon V_m\to TV_m$ such that $\langle x_i^m(x),x_j^m(x)\rangle_x=\delta_{ij}$ for all $x\in V_m$ and the scalar product $\langle\cdot,\cdot\rangle_x$ on $T_xM$.
	
	This is possible to do e.g.\@ by choosing the `usual' frame $\left(\frac{\partial\phi_m}{\partial x_i}\right)_{i=1}^n$ given by the smooth chart $\phi_m$, and then applying Gram-Schmidt to the resulting vectors in each $T_xM$. Since the pointwise scalar product of two $\mathcal C^\infty$ functions is again $\mathcal C^\infty$ and pointwise addition and scalar multiplication as well as the pointwise normalization of a nowhere-vanishing vector field likewise preserve smoothness, the ``Gram-Schmidt normalization operator'' $G$ maps $n$-tuples of smooth vector fields which are pointwise linearly independent, to $n$-tuples of smooth vector fields which pointwise form an orthonormal basis of their respective tangent space. Thus, $(x_i^m)_{i=1}^n:=G\left(\left(\frac{\partial\phi_m}{\partial x_i}\right)_{i=1}^n\right)$ is an orthonormal frame of $V_m$.
	
	Then, define $x_i\colon M\setminus N\to TM$ as $x_i|_{V_m}:=x_i^m$. Due to orthonormality, if $\tilde x_i$ denotes the cotangent vector dual to $x_i$, then $|\tilde x_1\wedge\dotsc\wedge\tilde x_n|=|\d V|$. W.l.o.g. choose the ordering of the $x_i$ such that this equality holds even without absolute values. Let $\chi\colon TM\to\R^n$ denote the `coordinate function' which maps $(x,\sum_{i=1}^n\alpha_i(x)x_i(x))\mapsto (\alpha_i(x))_{i=1}^n$ and which by orthonormality is an isometry on each $T_xM$.
	
	We are now ready to define \emph{integrability} of a section:
	
	\begin{definition}\label{def:integrable}
		A section $g\colon M\to TM$ is integrable with respect to the measure $\d V$ iff
		\begin{align*}
			\sum_m\int_{\phi_m(V_m)}\|\chi\circ g\circ\phi_m^{-1}\|\d\mu_m=:\int_M\|g\|\d V<\infty.
		\end{align*}
		
		Moreover, define $L^1(M,TM)$ to be the quotient space of integrable sections modulo $\d V$-almost-everywhere equivalence:
		\begin{align*}
			L^1(M,TM) &:= \{f\colon M\to TM\colon f\text{ integrable}\}/(f = g\ \ \d V\text{-a.e.\@}).
		\end{align*}
	\end{definition}
	
	We will of course want that this definition is independent of the choice of atlas.
	
	\begin{lemma}
		The notion of integrability introduced in Definition \ref{def:integrable} does not depend on the choice of $(V_m, \phi_m)$.
	\end{lemma}
	
	\begin{proof}
		Let $(X_\alpha,\psi_\alpha)$ be another patchwork of trivializations, let $(y_i)_{i=1}^n$ be the resulting orthonormal frame on their union, and let $\xi$ be the coordinate function analogous to $\chi$.
	
		It is sufficient to show equality on an intersection $D:=V_m\cap X_\alpha$. However, on such an intersection, we can indeed show that
		\begin{align*}
			\int_{\phi_m(D)}\|\chi\circ f\circ\phi_m^{-1}\|\sqrt{\det [g\left(x_i,x_j\right)]_{i,j}}\d\lambda = \int_{\psi_\alpha(D)}\|\chi\circ f\circ\psi_\alpha^{-1}\|\sqrt{\det [g\left(y_i,y_j\right)]_{i,j}}\d\lambda,
		\end{align*}
		because $h:=\|\cdot\|\circ\chi\circ f\colon M\to [0,\infty)$ is arbitrarily well approximable in $L^1$-norm by $\mathcal C^\infty(M)$ functions $h_k$ (since $h\circ\phi_m^{-1}$ is, and $\phi_m$ is a diffeomorphism). Note that this pertains to density of smooth functions within the space $L^1(M, \R)$, not $L^1(M,TM)$, and is thus a classical result.
		
		But for $h_k$ we know that $h_k\d V$ is simply some other classical $n$-form $\omega$, and clearly $n$-forms satisfy the stated transformation behaviour:
		\begin{align*}
			\int_{\phi_m(D)}h_k\circ\phi_m^{-1}\sqrt{\det[g(x_i,x_j)]_{i,j}}\d\lambda = \int_D\omega = \int_{\psi_\alpha(D)}h_k\circ\psi_\alpha^{-1}\sqrt{\det[g(y_i,y_j)]_{i,j}}\d\lambda.
		\end{align*}
		It follows that the equality also holds for the limit as $k\to\infty$, and that the limit can be brought inside the integrals by dominated convergence.
		
		Lastly, since both $\chi|_{\{x\}\times T_xM}$ and $\xi|_{\{x\}\times T_xM}$ are bijective isometries between $\R^n$ and $\{x\}\times T_xM$ with the same base point (both map $x$ to $0$), $\|\cdot\|\circ\chi$ and $\|\cdot\|\circ\xi$ agree on each $\{x\}\times T_xM$, and thus they agree on $TM$, finishing the proof.
	\end{proof}
	
	\begin{proposition}\label{prop:L1-Linfty-duality}
		The dual of the Banach space $L^1(M,TM)$ is $L^\infty(M,T^*M)$, and
		\begin{align*}
			\langle f,g\rangle_{(L^\infty,L^1)} := \int_M\langle f,g\rangle\d V
		\end{align*}
		defines a valid dual pairing on this pair of spaces.
	\end{proposition}
	
	\begin{proof}
		Via the isometry $L^1(M,TM)\equiv \bigoplus_{m\in\N}^{\ell^1}L^1(\phi_m(V_m),\R^n)$ (where the measure on $\phi_m(V_m)$ is $\mu_m$ from Equation \ref{eq:volume-measure-Rn}), we get that $L^1(M,TM)^*\equiv\bigoplus_{m\in\N}^{\ell^\infty}L^\infty(\phi_m(V_m),\R^n)$.
		
		However, clearly $L^\infty(\phi_m(V_m),\R^n)$ is isometric to $L^\infty(V_m,T^*M)$ (no matter the measure, as long as $\phi_m$ maps null sets to null sets), and
		\begin{align*}
			\sup_{m\in\N}\essup_{x\in V_m}\|f(x)\|_{T_x^*M}=\essup_{x\in M}\|f(x)\|_{T_x^*M},
		\end{align*}
		i.e.\@ $L^1(M,TM)^*\equiv L^\infty(M,T^*M)$.
	\end{proof}
	
	\begin{lemma} \label{lem:test-functions-dense}
		The compactly supported sections are dense in $L^1(M,TM)$.
	\end{lemma}
	
	\begin{proof}
		It is standard that all (locally compact) manifolds that are both Hausdorff and second-countable admit $\mathcal C^\infty$ partitions of unity. A proof of this fact can be found e.g.\@ as Corollary 3.4 in \cite[p.~33]{Lang}.
		
		Let $(\rho_k)_{k\in\N}$ be such a partition, constructed almost analogously to said Corollary 3.4, i.e.\@ we take the locally finite atlas $(\phi_m,U_m)_{m\in\N}$ and sets $V_m\subset U_m$ such that $\phi_m(U_m) = B_3(0)\subset\R^n$ and $\phi_m(V_m) = B_1(0)$, while the $V_m$ still cover $M$, i.e. $\bigcup_{m\in\N}V_m=M$. These sets and covering exist by the preceding Theorem 3.3 in \cite{Lang}.
		
		Thus, we can slightly modify the construction of $(\psi_k)_{k\in\N}$, such that $\psi_k$ is still constant $1$ on $V_m$, but now also constant $0$ outside of $W_m:=\phi_m^{-1}(B_2(0))$ (rather than outside the preimage of $B_3(0)$ as in the Corollary) and taking values in $[0,1]$ globally.
		
		Then we can clearly proceed the same way, by summing $\psi(x):=\sum_{k\in\N}\psi_k(x)$ (this sum is finite for each $x\in M$ due to local finiteness of our atlas), yielding a slightly modified locally finite partition of unity $(\rho_k)_{k\in\N}:=\left(\psi_k/\psi\right)_{k\in\N}$, now with the additional property that all the $\rho_k$ are compactly supported, as all the $\overline{W_m}$ are compact (they are the images of compact sets under a continuous mapping).
		
		But now we can simply approximate any function $g\in L^1(M,TM)$ by the sequence of compactly supported $L^1$-sections $\left(g\left(\sum_{i=1}^k\rho_i\right)\right)_{k\in\N}$, as the partition of unity gives us that
		\begin{align*}
			\int_M \left\|g-g\left(\sum_{i=1}^k\rho_i\right)\right\|\d V
			&= \int_M \left\|g\left(\sum_{i\in\N}\rho_i-\sum_{i=1}^k\rho_i\right)\right\|\d V = \int_M \left\|g\left(\sum_{i>k}\rho_i\right)\right\|\d V\\
			&=\int_{\bigcup_{i>k}W_i}\left\|g\left(\sum_{i>k}\rho_i\right)\right\|\d V \leq \int_{\bigcup_{i>k}W_i}\|g\|\d V\overset{k\to\infty}\longrightarrow0.\qedhere
		\end{align*}
	\end{proof}
	
	\begin{lemma}\label{lem:smooth-compact}
		The smooth compactly supported sections are dense in the compactly supported $L^1(M,TM)$ sections.
	\end{lemma}
	
	\begin{proof}
		By construction of the Lebesgue measure on $\R^n$, $L^1(\R^n,\R^n)$ functions are the limit of simple functions. And simple functions can be arbitrarily well approximated in $L^1$-norm by smooth functions.
		
		So for each $g\in L^1(K,TK)$ for some compact set $K\subset M$, and every $\e>0$, there exists a finite atlas $(\phi_m,U_m)$ of $K$ with $L$-Lipschitz smooth charts $\phi_m$ such that, if we set $W_m:=\phi^{-1}(U_m) \subset \R^n$, then $g|_{U_m}\circ \phi_m$ is a $L^1(W_m,TU_m)$ section which, if we choose our atlas fine enough, will correspond isometrically to an $L^1(W_m,\R^n)$ section (via our choice of local coordinates), and thus it is $\e$-close to a $\mathcal C^\infty(W_m,\R^n)$ section $h$. Clearly, $h\circ\phi_m^{-1}\colon U_m\to TU_m$ then lies in $\mathcal C^\infty(U_m,TU_m)$ and is $L\e$-close to $g|_{U_m}$, and thus overall, using a smooth partition of unity, there exists a $\mathcal C^\infty(K,TK)$ function that is $m_0L\e$-close to $g$ (where $m_0$ is the number of charts in our atlas).
	\end{proof}

	\section{Isometric Representation of \texorpdfstring{$\Lip_0(M)$}{Lip0(M)}}
	
	\begin{theorem}\label{prop:range-of-derivative}
		For any $F\in\Lip_0(M)$ let $D\colon \Lip_0(M)\to L^\infty(M,T^*M)\colon F\mapsto\d F$. Then, the following hold:
		\begin{enumerate}[i)]
			\item $D$ is a linear isometry of $\Lip_0(M)$ into $L^\infty(M,T^*M)$.
			\item The range of $D$ are the exact $L^\infty$-currents:\label{prop:range-of-derivative:ii}
			\begin{align*}
				Z(M)=\{f\in L^\infty(M,T^*M)\colon \exists g\in \Omega^0(M)^\diamond\colon \d g=f\}.
			\end{align*}
		\end{enumerate}
	\end{theorem}
	
	\begin{proof}
		\begin{enumerate}[i)]
			\item Let $F\in\Lip_0(M)$, $L:=\Lip(F)$, $K:=\|DF\|_\infty$. We know from Proposition \ref{prop:Elementary_Calculus}.\ref{list:Elementary_Calculus1}) that $K\leq L$, so assume that $K < L$ were true, and let $\e_1:=(L-K)/2$. By definition of $L$ we can find points $x, y\in M$ such that $F(x)-F(y) \geq (L-\e_1)d(x,y)$. Since $M$ is a length space, we can even find such points arbitrarily close to each other, this follows e.g.\@ from Lemma 3.4 in \cite{BDR}.
			
			Recall that the exponential map $\exp_x\colon U\subset T_xM\to M$ maps a tangent vector $v\in U$ to $\gamma_v(1)$, where $\gamma_v$ is the unique geodesic satisfying $\gamma_v(0)=x$ and $\gamma_v'(0)=v$. Since $\exp_x$ is a smooth map whose Jacobian at $0\in U^\circ$ is the identity, we can restrict $\exp_x$ to a neighbourhood of $0$ small enough that $exp_x$ is a diffeomorphism and both it and its inverse have Lipschitz constant arbitrarily close to $1$.
			
			So assume w.l.o.g.\@ that $\exp_x$ is defined on such a neighbourhood $\tilde U$, that it is a $(1+\e_2)$-isometry for some $\e_2>0$, and that $y$ is in the image of $\exp_x$. If we now choose a basis in $T_xM$, i.e.\@ if we map it to $\R^n$ with the smooth isometry $\iota$, then it is a classical fact that $\Lip(\tilde F)=\|\d\tilde F\|_\infty$ for the function $\tilde F:=F\circ\exp_x\circ\iota\colon \R^n\to\R$, shown via convolution as in \cite{CKK2016} or \cite{F2024}.
			
			Thus, we now know that
			\begin{align*}
				L&=\Lip(F)= \Lip(\tilde F\circ\iota^{-1}\circ\exp_x^{-1}) \leq \Lip(\tilde F)(1+\e_2)=\|\d\tilde F\|_{L^\infty(\R^n,\R^n)}(1+\e_2).
			\end{align*}
			Note at this point that $\tilde F$ is also the pullback $\tilde F = (\exp_x\circ\iota)^* (F)$, that the exterior derivative commutes with pullbacks even in the distributional setting $\d\tilde F=(\exp_x\circ\iota)^*(\d F)$ \cite[Thm.~3.1.23]{GKOS}, and that the pullback $(\exp_x\circ\iota)^*$ of a $1$-form is a linear map that has operator norm $\|(\exp_x\circ\iota)^*\|\leq\Lip(\exp_x\circ\iota)$ (due to smoothness). So, in total,
			\begin{align*}
				L\leq \|(\exp_x\circ\iota)^*(\d F)\|_{L^\infty(\R^n,\R^n)}(1+\e_2) \leq K(1+\e_2)^2.
			\end{align*}
			Clearly, there exists a valid choice of $\e_2>0$ that is small enough such that this is contradictory since $K < L$.
			
			\item By definition, every function value $D F$ lies in $Z(M)$.
			
			On the other hand, if we have an exact $1$-form $\d g\in Z(M)$, use Lemma \ref{lem:distribution-is-abs-cont-function} to show that $g$ is absolutely continuous. Then, we can use the fundamental theorem of calculus to get that $g$ must be Lipschitz due to the a.e.\@ global bound on its derivative $\d g\in Z(M)\subset L^\infty(M,T^*M)$.\qedhere
		\end{enumerate}
	\end{proof}
	
	\begin{corollary}\label{cor:closed-currents}
		If $M$ is simply connected, then $Z(M)$ is equivalently also the space of all closed essentially bounded currents, $Z(M)=\{f\in L^\infty(M,T^*M)\colon \d f=0\}$.
	\end{corollary}
	
	\begin{proof}	
		This follows from the fact that if $M$ is simply connected, then every closed distributional $1$-form is locally exact (see \cite[Thm.3.1.30]{GKOS}), and, as we will see, is arbitrarily well approximable by an exact form $\d A_\e\omega$ plus (or minus) a smooth form $R_\e\omega$ (which is closed as the difference of two closed forms). But for smooth forms, it is a well-known classical result that simply connected implies that all closed $1$-forms on $M$ are exact.
		
		More precisely: there exist nets of operators $R_\e,A_\e$ ($R_\e$ mapping $k$-forms to $k$-forms and $A_\e$ mapping $k$-forms to $(k-1)$-forms) such that for each distributional 1-form $\omega\in\Omega^1(M)^\diamond$, there exists an $\e_0(\omega)$ such that for all $\e<\e_0$, $R_\e\omega$ is a classical $1$-form and:
		\begin{align*}
			\omega = R_\e\omega + \d A_\e\omega - A_\e(\d\omega).
		\end{align*}
		It follows by applying $\d$ to both sides of this equation and setting $\d\omega=0$ that $R_\e\omega$ is closed if $\omega$ is. Thus, by simple connectedness of $M$, it is exact, and therefore there exists a smooth 0-form $\tau$ such that $\d\tau = R_\e\omega$, implying that for closed forms $\omega$
		\begin{align*}
			\omega &= d(\tau + A_\e\omega).\qedhere
		\end{align*}
	\end{proof}
	
	\section{Constructing the Pre-Adjoint}
	This last section proves that the operator $D$ we've been working with is, in fact, the adjoint of some other operator ${}^*\!D$, which will serve as our isometry between the Lipschitz-Free Space $\F(M)$ and the quotient $L^1(M,TM)/Z_\perp(M)$.
	
	\begin{proposition}\label{prop:wtow}
		The isometry $D$ is weak$^*$-to-weak$^*$ continuous.
	\end{proposition}
	
	\begin{proof}
		Let $A$ be a bounded subset of $\Lip_0(M)$, then by Theorem 2.37 in \cite{Weaver}, weak$^*$ convergence on $A$ coincides with pointwise convergence, i.e.\@ $(f_k)_{k\in\N} \overset{w^*}\to f$ iff for all $x\in M$, $f_k(x)\to f(x)$. Since $A$ is bounded with respect to the Lipschitz norm, assume all $f_k$ and $f$ are at most $L$-Lipschitz, giving us uniform convergence on compact sets.
		
		On the other hand, weak$^*$ convergence of $(\d f_k)_{k\in\N}$ to $\d f$ in $L^\infty(M,T^*M)$ by definition means that for all functions $g\in L^1(M,TM)$, $\langle \d f_k,g\rangle_{(L^\infty,L^1)}\to\langle \d f,g\rangle_{(L^\infty,L^1)}$. Since $\d f_k$ is a bounded sequence, it is sufficient to check the above condition for all $g$ that belong to a dense subset of $L_1(M,TM)$. Thus, by Lemma \ref{lem:smooth-compact}, it is enough to prove that $\langle\d f_k,g\rangle_{(L^\infty,L^1)}\to\langle\d f,g\rangle_{(L^\infty,L^1)}$ for all $g\in L^1(M,TM)$ that are smooth and compactly supported. So let $g$ be such a section and let $K$ be its support. Then,
		\begin{align*}
			\langle\d f_k,g\rangle_{(L^\infty,L^1)} = - \int_K\div(g)\d V\wedge f_k
		\end{align*}
		Since $K$ is compact, $f_k\to f$ uniformly on $K$, whereas $\div(g)\d V$ is a smooth $n$-form. Therefore, it is bounded, hence $\div(g)\d V\wedge f_k\to \div(g)\d V\wedge f$ uniformly on $K$, proving $\langle\d f_k,g\rangle_{(L^\infty,L^1)}\to \langle\d f,g\rangle_{(L^\infty,L^1)}$ as $k\to\infty$ for all $(f_k)_{k\in\N}$ converging weak$^*$ to $f$.
	\end{proof}

	Recall that by test functions $\phi$, and the space of test functions $\mathcal D(M)$, we mean smooth, scalar-valued ($\mathcal C^\infty(M)$) functions with compact support in $M$. And that the (classical) Lie derivative $L_g$ for a (smooth) vector field $g$ maps any $k$-form $\omega$ to the $k$-form representing ``the derivative of $\omega$ along the (direction of) flow of $g$'', in the sense that it is characterised by the following three axioms:
	\begin{enumerate}
		\item Any $0$-form $f\in \mathcal C^\infty(M)$ is mapped to its directional derivative in the direction of $g$: $L_gf:=\nabla_gf\in \mathcal C^\infty(M)$
		\item For all smooth forms $\sigma, \tau$, we have $L_g(\sigma\wedge\tau)=L_g(\sigma)\wedge\tau+\sigma \wedge L_g(\tau)$.
		\item $L_g$ is $\R$-linear and $\d L_g=L_g\d$.
	\end{enumerate}

	\begin{definition}
		Let $\d V\in\Omega^n(M)$ be our volume form and $g\in L^1(M,TM)$ an integrable distributional vector field, then $\div(g)\in\mathcal D'(M)$ is defined as
		\begin{align*}
			\div(g)(\d V):=L_g(\d V).
		\end{align*}
		The proof that these classical definitions of the divergence and Lie derivative $L_g$ extend to distributional vector fields can be found in \cite[Thm.~3.1.41]{GKOS}.
	\end{definition}
	
	\begin{lemma}\label{lem:divergence-to-dual-pairing}
		For the volume form $\d V=\d x^1\wedge\dotsb\wedge\d x^n$, a test function $\phi\in\mathcal D(M)$ and an integrable distributional vector field $g\in L^1(M,TM)$, we have
		\begin{align*}
			\div(g)(\d V)(\phi)=-\int_M\langle \d \phi,g\rangle\d V.
		\end{align*}
	\end{lemma}
	
	\begin{proof}
		This fact can also be found in \cite{GKOS} in the paragraph below Thm.~3.1.42. Since the proof is so short, it is included here for completeness.
		
		We know from \cite[Thm.~3.1.41(iv)]{GKOS} that the distributional Lie derivative satisfies the Leibniz rule, thus:
		\begin{align}\label{eq:divergence}
			\div(g)(\d V)=L_g(\d V)=\sum_{i=1}^{n}\d x^1 \wedge \dotsb \wedge L_g(\d x^i) \wedge \dotsb \wedge \d x^n.
		\end{align}
		On the other hand, the distributional Lie derivative still satisfies the classical identity $L_g\omega=\d i_g\omega + i_g\d\omega$ (see \cite[Thm.~3.1.25(iii)]{GKOS}). Thus, for $\omega:=\d x^i$, we get:
		\begin{align*}
			L_g(\omega) = (\d i_g\omega) + (i_g\d\omega) = \d(\langle \d x^i,g\rangle) + 0 = \sum_{j=1}^n\frac{\partial g_i}{\partial x^j}\d x^j.
		\end{align*}
		Here $g_i$ is the $i$-th component of $g$ in local coordinates, thus it is a distribution in $\mathcal D'(M)$.
		
		So, since the cross-terms where $j\neq i$ cancel when we substitute this expression for $L_g(\d x^i)$ into equation (\ref{eq:divergence}), as expected, $\div(g)(\d V)=\sum_{i=1}^n\frac{\partial g_i}{\partial x^i}\d V$ in local coordinates.
		
		From there it is clear by partial integration in each of the $n$ summands that on test functions $\phi$,
		\begin{align*}
			\div(g)(\d V)(\phi)&=\int_M\sum_{i=1}^n\frac{\partial g_i}{\partial x_i}\phi\d V = - \sum_{i=1}^{n} \int_Mg_i\frac{\partial \phi}{\partial x_i}\d V = -\int_M\langle \d \phi,g\rangle\d V.\qedhere
		\end{align*}
	\end{proof}
	
	We are now ready to define the pre-annihilator of $Z(M)$, the space of all exact essentially bounded currents $Z(M)=\{f\in L^\infty(M,T^*M)\colon\exists g\in\Omega^0(M)^\diamond\colon\d g=f\}$:
	\begin{definition}
		Let $I(M)$ denote the divergence-free integrable sections, and recall that the pre-annihilator of a set $S$ is the set of all elements of the predual that vanish on all elements of $S$:
		\begin{align*}
			I(M)&:=\{g\in L^1(M,TM)\colon \div(g)(\d V)=0 \in \mathcal D'(M)\}.\\
			Z_\perp(M)&:=\{g\in L^1(M,TM)\colon \forall F\in \Lip_0(M)\colon \langle \d F,g\rangle_{(L^\infty,L^1)} = 0\}.
		\end{align*}
	\end{definition}
	
	\begin{lemma}
		The pre-annihilator $Z_\perp(M)$ is contained in the divergence-free integrable sections $I(M)$.
	\end{lemma}
	
	\begin{proof}
		Note that the test functions are a subset of the set of all Lipschitz functions on $M$: $\mathcal D(M) \subset \Lip(M)$. And the condition $\langle \d F, g \rangle_{(L^\infty,L^1)} = 0$ for all $F\in \Lip_0(M)$ implies that also $\langle \d (f - f(0)) + 0, g \rangle_{(L^\infty,L^1)} = 0$ for all $f\in\Lip(M)$, and thus $\langle \d f, g\rangle_{(L^\infty,L^1)} = 0$ for all $f\in\mathcal D(M)$, implying $\div(g)(\d V)=0$ when evaluated against all test functions.

		However, that means that by Lemma \ref{lem:divergence-to-dual-pairing}, for all evaluations against test functions $\phi$, $\int_M\div(g)(\d V)\wedge \phi = -\int_M\langle\d \phi,g\rangle\d V = 0$, and thus $\div(g)(\d V)=0\in\mathcal D'(M)$.
		
		It follows that $Z_\perp(M)\subseteq I(M)$.
	\end{proof}
	
	We now give two important cases where we have equality between these two sets, together covering every case where $M$ is compact (without boundary):
	
	\begin{proposition}\label{prop:complete-equality}
		If $M$ is complete and unbounded, then the pre-annihilator $Z_\perp(M)$ of $Z(M)$ is the set of divergence-free integrable sections: $I(M)=Z_\perp(M)$.
	\end{proposition}
	
	\begin{proof}
		Let $g\in I(M)$. Then, we know that for each test function $\phi\in\mathcal D(M)$,
		\begin{align*}
			\int_M\langle \d \phi,g\rangle\d V = - \int_M\div(g)(\d V)\wedge \phi = 0.
		\end{align*}
		So now let $h\colon \R\to\R$ be a function which satisfies $h(t)=1$ for all $t\leq0$, $h(t)=0$ for all $t\geq1$, $\|h'\|_\infty\leq 2$ and $h\in\mathcal C^\infty(\R)$.
		Then, for $k\in\N$, define
		\begin{align*}
			h_k&\colon M\to \R\colon x\mapsto h\left(\frac{d(0_M,x)}k-1\right)\\
			g_k(x)&:=g(x)h_k(x).
		\end{align*}
		Then,
		\begin{align*}
			\div (g_k)(\d V) &= h_k\div (g)(\d V) + L_gh_k\wedge\d V\\
			L_gh_k(x)&=\langle g,\nabla\rangle h_k = \sum_{i=1}^n g^i(x)h'\left(\frac{d(0,x)}k-1\right)\frac{1}k\frac{\partial d(0,\cdot)}{\partial x^i}(x).
		\end{align*}
		Analogous to the proof of Lemma \ref{lem:test-functions-dense}, we get that $g_k\to g$ in $L^1$ as $k\to\infty$.
		
		On the other hand, since $d(0,x)$ is $1$-Lipschitz in $x$, all partial derivatives exist almost everywhere and are bounded by $1$. Moreover, $h'$ will be 0 everywhere outside of $[0,1)$, which happens iff $x\in B_{2k}\setminus B_k$. It follows that
		\begin{align*}
			|L_gh_k(x)|&\leq \sum_{i=1}^{n}|g^i(x)|\|h'\|\chi_{B_{2k}\setminus B_k}(x) \frac1k \cdot1 \leq\frac{2\sqrt n}k\|g(x)\|_{T^*M}\chi_{B_{2k}\setminus B_k}(x).
		\end{align*}
		
		Thus, thanks to $g_k$ having compact support, we get by partial integration that, for any $f\in\Lip_0(M)$, (and keeping in mind that $\div(g)(\d V)=0$ by definition of $I(M)$):
		\begin{align*}
			\left|\int_M\langle \d f,g_k\rangle\d V\right|
			&= \left|\int_M \div(g_k)(\d V)\wedge f \right|
			\leq \frac{2\sqrt n}k\int_{B_{2k}\setminus B_k}\|g\|_{T^*M}|f|\d V\\
			&\leq \frac{2\sqrt n}k\int_{B_{2k}\setminus B_k}\|g\|_{T^*M}2k\Lip(f)\d V\\
			&\leq {4\sqrt n}\Lip(f)\int_{M\setminus B_k}\|g\|_{T^*M}\d V \overset{k\to\infty} \longrightarrow 0.
		\end{align*}
		
		Together with $\|g_k-g\|_{L^1}\to0$, this gives us the sought-after equation
		\begin{align*}
			\left|\int_M\langle\d f,g\rangle\d V \right| \leq \Lip(f)\|g_k-g\|_1 + \left|\int_M\langle\d f,g_k\rangle\d V\right| \overset{k\to\infty}\longrightarrow 0
		\end{align*}
		for all $f\in\Lip_0(M)$.
	\end{proof}
	
	\begin{proposition}\label{prop:compact-equality}
		If $M$ is compact, then $I(M)=Z_\perp(M)$.
	\end{proposition}
	
	\begin{proof}
		Let $g\in I(M)$. Then, since $\mathcal D(M)=\mathcal C^\infty(M)$ and $\partial M=\emptyset$, we know that
		\begin{align*}
			\int_M\langle\d\phi,g\rangle\d V = \int_{\partial M} \phi g\d V - \int_M \div(g)(\d V)\wedge\phi = 0
		\end{align*}
		for all $\phi\in\mathcal C^\infty(M)$.
		
		Now assume that $f\in\Lip_0(M)$ with $\Lip(f)=L$. Then, for every $r>0$ there exists a sequence of $\mathcal C^\infty(M)$-functions $(f_n)_{n\in\N}$ which satisfies $\Lip(f_n)\leq\Lip(f)+r$ and, for all $p\in M$, $|f(p)-f_n(p)|\leq 2^{-n}$ (see e.g.\@ \cite{AFLR}).
		
		Thus, $(f_n)_{n\in\N}$ is a bounded sequence in $\Lip_0(M)$ and it converges to $f$ pointwise. And on bounded subsets of $\Lip_0(M)$, pointwise convergence is equivalent to weak$^*$ convergence, hence $(f_n)_{n\in\N}$ converges weak$^*$ to $f$.
		
		Since $D$ is weak$^*$-to-weak$^*$ continuous by Prop \ref{prop:wtow}, $(Df_n)_{n \in \N}$ weak$^*$ converges to $Df$.
		
		This by Proposition \ref{prop:L1-Linfty-duality} and the definition of weak$^*$ convergence means that the evaluation $\langle Df,g\rangle_{(L^\infty,L^1)}$ of $Df$ against any $L^1(M,TM)$ section $g$ is equal to the limit of the evaluations $\langle Df_n,g\rangle_{(L^\infty,L^1)}$ as $n\to\infty$. But since the $f_n$ were chosen to lie in $\mathcal C^\infty(M)$, we have already shown that if $g\in I(M)$, then $\langle Df_n,g\rangle_{(L^\infty,L^1)}=0$, thus so does their limit, proving that $g\in Z_\perp(M)$.
	\end{proof}
	
	\begin{corollary}\label{cor:domain-inside-complete}
		Let $M$ be a domain inside a complete orientable Riemannian manifold $N$. Moreover, let $I(M,N)$ be the set of sections that are integrable over $M$ and divergence-free over $N$ when extended to $N$ by zero:
		\begin{align*}
			I(M,N):=\{g\in L^1(M,TM)\colon \exists\hat g\in L^1(N,TN)\colon&\hat g|_M=g,
			\hat g|_{N\setminus M}=0,\\
			&\div(\hat g)(\d V)=0\in \mathcal D'(N)\}.
		\end{align*}
		Then, $I(M,N)=I(M)=Z_\perp(M)$.
	\end{corollary}
	
	\begin{proof}
		Trivially, $Z_\perp(M)\subseteq I(M)\subseteq I(M,N)$.
		
		For the other direction, we will use the same approach as in the last part of the proof of Proposition~3.3 in \cite{CKK2016} by Cúth, Kalenda and Kaplick\'y.
		
		Their proof, reproduced here for completeness, goes as follows: choose $g\in I(M,N)$, and let $F\in\Lip_0(M)$. Moreover, let $\tilde F\in\Lip_0(N)$ be any extension of $F$, e.g.\@ the one that exists by MacShane. Then,
		\begin{align*}
			\langle \d F,g\rangle_{(L^\infty,L^1)} &= \int_M\langle \d F,g\rangle \d V = \int_M\langle \d \tilde F,g\rangle \d V\\
			&= \int_N \langle \d \tilde F,\hat g\rangle \d V = \langle \d \tilde F,\hat g\rangle_{(L^\infty,L^1)} = 0.
		\end{align*}
		The last step uses either Proposition \ref{prop:complete-equality} or \ref{prop:compact-equality}, depending on the boundedness of $N$.
	\end{proof}
	
	\begin{theorem}\label{thm:isometry-of-lip-free}
		The Lipschitz-Free Space $\F(M)$ of a complete connected orientable Riemannian manifold $M$ is isometric to $L^1(M,TM)/I(M)$, the space of equivalence classes of integrable sections with regards to the relation $g\sim h \Leftrightarrow \div(g-h)=0$.
	\end{theorem}
	
	\begin{proof}
		Since $D$ is an isometric weak$^*$-to-weak$^*$ continuous operator, it has an isometric pre-adjoint ${}^*\!D$ mapping elements of a predual of $Z(M)$ to elements of a predual of $\Lip_0(M)$. In this case, with the dual pairings we have chosen, the predual of $Z(M)$ in question is $L^1(M,TM)/Z_\perp(M)=L^1(M,TM)/I(M)$, and the corresponding predual of $\Lip_0(M)$ is $\F(M)$.
	\end{proof}
	
	\begin{remark}\label{rem:comparison-domain-results}
		Using Corollary \ref{cor:domain-inside-complete}, one also gets that $\mathcal F(M) = L^1(M,TM)/I(M,N)$ in the case that $M$ is a connected Riemannian manifold that lies isometrically as a domain inside a complete Riemannian manifold $N$.
		
		Like in \cite{CKK2016} and \cite{F2024}, the equivalence $Z_\perp(M)=I(M)$ requires us to somehow `see' whether the integrable sections on the boundary behave well enough to be extended by zero outside of $M$ in a larger complete space, and still have divergence zero on said larger space. In a manifold $M$, such `outside' structure need not exist, hence the additional requirement that it be a submanifold of a complete space.
	\end{remark}
		
	Lastly, note that we crucially require connectedness to make $M$ a `genuine' metric space without infinite distances, so that we can meaningfully talk about $\Lip_0(M)$ and $\F(M)$.
	
	It might be possible to adapt the proofs in this paper to also apply to more general types of connected smooth manifolds, for example (symmetric, orientable) Finsler manifolds, as Flores' result \cite{F2024} and this Remark both provide independent partial answers in the positive. The underlying theory of distributional forms developed in \cite{GKOS} certainly holds for all orientable paracompact smooth Hausdorff manifolds.
	
	\subsection*{Acknowledgments} 
	The author wishes to thank their advisors Christian Bargetz and Tommaso Russo for helpful discussions and suggestions to improve this paper, and Birgit Schörkhuber for relevant pointers, as well as the reviewer for their detailed report on what to improve, and their helpful suggestions on how to do that.
	
	The author gratefully acknowledges the support by the Vice-Rectorate for Research at the University of Innsbruck in the form of a PhD scholarship.
	
	\subsection*{Data Availability}
	This paper does not rely on any outside data.
	

\end{document}